\documentclass[12pt]{article}

\usepackage{color}
\usepackage{amsmath,amssymb,amsthm}
\usepackage{wrapfig}
\usepackage{ifpdf}

\newtheorem{thm}{Theorem}[section]
\newtheorem{cor}{Corollary}[section]
\newtheorem{lemma}{Lemma}[section]
\newtheorem{prop}{Proposition}[section]
\newtheorem{defn}{Definition}[section]

\newcommand{\lw}{\vartriangleright}

\newcommand{\rk}{\mathop{\rm rank}\nolimits}

\title {Majorization and minimal energy on spheres} 

\author {Oleg R. Musin}

\begin{document}

	
\date{}
\maketitle

\begin{abstract} 

\end{abstract}
 In the present paper, we consider the majorization theorem (also known as Karamata's inequality) and the respective minima of the majorization (the so-called $M$-sets) for $f$-energy potentials of $m$-point configurations on the unit sphere.  In particular, we show the optimality of regular simplexes, describe some $M$-sets of small cardinality, define and discuss spherical $f$-designs. 

\medskip


\section{Introduction}

\medskip

Let $A=(a_1,\ldots,a_n)$ be an arbitrary sequence of real numbers. Let $A_\uparrow=(a_{(1)},\ldots,a_{(n)})$ denote the sequence obtained from $A$ by arranging its elements in the non-decreasing order: $a_{(1)}\le a_{(2)}\le\ldots\le a_{(n)}.$

Given two sequences $A=(a_1,\ldots,a_n)$ and $B=(b_1,\ldots,b_n),$ we say that  {\em $A$ majorizes $B$}, and write $A\lw B$ if the following conditions are fulfilled:
$$
a_{(1)}+\ldots+a_{(k)}\ge b_{(1)}+\ldots+b_{(k)}, \,\, k=1,\ldots,n.
$$

\noindent{\bf Remark.} Note that in \cite{HLP} and \cite{MO} this condition is called a {\it weak majorization}. 

The main theorem in the theory of majorization is the majorization (or Karamata) inequality (see details in \cite{HLP,MO}). Here, we shall consider its ``weak'' version. 

\medskip

\noindent{\bf Theorem (The majorization theorem).} { \it Let $f(x)$ be a convex decreasing function, and let $A = (a_1, \ldots, a_n)$, $B = (b_1, \ldots, b_n)$ be two sequences. Then if $A\lw B$, we have 
$$
f(a_1)+\ldots+f(a_n)\le f(b_1)+\ldots+f(b_n). \eqno (1.1)
$$
Moreover, $A\lw B$ if and only if for all convex decreasing functions $g$ we have 
$$
g(a_1)+\ldots+g(a_n)\le g(b_1)+\ldots+g(b_n).
$$

Provided $f(x)$ is strictly convex, the inequality $(1.1)$ turns into equality if and only if we have $A_\uparrow = B_\uparrow$. 
}

\medskip

 Let $P$ be a set of sequences of length $n$. {\em We say  that $A\in P$ is an $M$-set if for any $B\in P$ either $A\lw B$ or  $A$ and $B$ are incomparable.} Let $M(P)$ denote the set of all $M$-sets in $P$.

 Suppose that $f(x)$ be a strictly convex decreasing function, and define $$E_f(A):=f(a_1)+\ldots+f(a_n), \; A=(a_1,\ldots,a_n).$$
 Then the majorization theorem readily implies that if $E_f$ achieves its minimum at some $A\in P$, then $A\in M(P)$. 

\medskip

 Let $S$ be a set and $\rho:S\times S\to {\Bbb R}$ be a symmetric function. Let $X$ be a subset of $S$ of cardinality $m$ (or, shortly, an $m$-subset).  Denote by $R_\rho(X)$  the set of values of $\rho(a,b)$ over all unordered pairs $(a,b)$ of elements in $X$. Then we have the following generalization of the majorization theorem (Theorem 2.1):
 
 \medskip
 
 \noindent{\bf Theorem.} { \it
Let $X$ and $Y$ be two $m$-subsets of $S$. Suppose $R_\rho(X)\lw R_\rho(Y)$. Then for every convex decreasing function $f$ we have $E_f(X)\le E_f(Y)$. }

 \medskip
 
  Now, let us put $P=\{R_\rho(X)\}$, where $X\subset S$ is an $m$-subset. Then we define $M(S,\rho,m)=M(P)$ and discuss its properties further in Section 2. Our main interest is the case when  $S$ is the standard unit sphere ${\Bbb S}^{n-1}$ in ${\Bbb R}^{n}$. In particular, we prove that on the unit circle ${\Bbb S}^1$ with angular distance  $\rho=\varphi$, the $M$-sets are vertex sets  of regular polygons (Theorem \ref{T32}). We also show that $M({\Bbb S}^{n-1},\rho,n+1)$ consists of vertices of regular simplices, for $\rho$ being the standard Euclidean distance squared (Theorem \ref{Tsimplex}), and  describe some small $M$-sets, and some are left as open problems, depending on the ``distance fucntional'' $\rho$. 
  In Section 6 we define spherical $f$-design and study their properties.  Then we discuss  possible relations between the notions of $f$-designs and $M$-sets (Theorem \ref{T42}), $\tau$-designs, and two-distance sets. 


\medskip

\section{$M$-sets and minimums of potential energy}

Let $S$ be an arbitrary set, and $\rho: S\times S\to D\subset {\Bbb R}$ be any symmetric function. Then for a given convex decreasing function $f:D\to {\mathbb R}$  and for every finite subset $X=\{x_1,\ldots x_m\} \subset S$ we define the potential energy $E_f(X)$ of $X$ with respect to $f$ as
$$
E_f(X):=\sum\limits_{1\le i<j\le m}{f(\rho(x_i,x_j))}.
$$

In this paper we consider the following minimum energy problem.

\medskip

\noindent{\bf Generalized Thomson's Problem.} {\it For $S,\rho, f$ and $m$ given, find all $X\in S^m=S\times...\times S$ such that $E_f(X)$ is the minimum of $E_f$ over the set of all $m$-subsets of S.}

\medskip 

Let $R_\rho(X)$ denote the set of all $\rho(x_i, x_j)$, where $1\le i<j\le m$, i.e.  
$$
R_\rho(X):=\{\rho(x_1,x_2)\ldots,\rho(x_1,x_m),\ldots,\rho(x_{m-1},x_m)\}.
$$

Then the majorization theorem implies 
\begin{thm}\label{MajI} Let $X$ and $Y$ be two $m$-subsets of $S$. Suppose $R_\rho(X)\lw R_\rho(Y)$. Then for every decreasing and convex function $f$ we have $E_f(X)\le E_f(Y)$. 
\end{thm}

Note that $\lw$ defines a partial order on $S^m$. Let $X$ and $Y$ be two $m$-subsets of $S$, such that $R_\rho(X)\ne R_\rho(Y)$. Then we have one of three following cases:  either $R_\rho(X)\lw R_\rho(Y)$, or $R_\rho(Y)\lw R_\rho(X)$, or $R_\rho(X)$ and $R_\rho(Y)$ are incomparable. Now we define the maximal subsets of the poset $(S^m, \lw)$. 

\begin{defn} \label{def21} We say that  $X\in S^m$ is an {$M$-set in $S$ with respect  to $\rho$} if for any $Y\in S^m$ we have that either $R_\rho(X)\lw R_\rho(Y)$, or $R_\rho(X)$ and $R_\rho(Y)$ are incomparable. Let $M(S,\rho,m)$ denote the set of all $M$--sets in $S$ of cardinality $m$.   
\end{defn}

\begin{defn} \label{def22}  Let $f:D\to {\Bbb R}$ be a convex decreasing function. Let  $V_f=\inf_{Y\in S^m}{E_f(Y)}$. Let $M_f(S,\rho,m)$ denote the set of all $X\in S^m$ such that $E_f(X) = V_f$.
\end{defn}

\begin{thm} \label{T2}  Let $S$ be a compact topological space and $\rho: S\times S\to D\subset {\Bbb R}$ be a symmetric continuous function.  
Let $f: D\to {\Bbb R}$ be a strictly convex decreasing function. Then $M_f(S,\rho,m)$ is non-empty and $M_f(S,\rho,m)\subseteq M(S,\rho,m)$.  
\end{thm}

\begin{proof}  Since $S$ is compact, there is $X\in S^m$ such that $E_f(X)=V_f$, i.e. $M_f(S,\rho,m)\ne\emptyset$.  The majorization  theorem yields that for all $Y\in S^m$ we have either $R_\rho(X)\lw R_\rho(Y)$ or  $R_\rho(X)$ and $R_\rho(Y)$  are incomparable. Indeed, if $X,Y\in M_f(S,\rho,m)$ while $R_\rho(X)$ and $R_\rho(Y)$  are comparable, then $R_\rho(X)=R_\rho(Y)$. This means that $X$ is an $M$--set. 
\end{proof}

\noindent{\bf Remark.} I am very grateful to one of the reviewers who pointed out an inaccuracy in the previous version of the manuscript. Namely, if we do not require $f$ be strictly convex, then the following counterexample shows that Theorem \ref{T2} cannot hold. 

Let $S= \mathbb S^1$, $\rho=\varphi$, where $\varphi$ is the angular (geodesic) distance, and $f(t) = -t$. Clearly, $f$ is convex and decreasing, though not \textit{strictly} convex. It is easy to see that $M_f({\Bbb S}^1,\varphi,m)$, for $m$ even, contains all centrally symmetric sets $X$ in  $\mathbb S^1$. However, $M({\Bbb S}^1,\varphi,m)$ consists of the vertices of regular $m$--gons inscribed in $\mathbb{S}^1$, see Theorem \ref{T32}. Obviously, $M_f({\Bbb S}^1,\varphi,m)\nsubseteq M({\Bbb S}^1,\varphi,m)$. 

\begin{thm} \label{T_3} Let $\rho:S\times S\to D\subset {\Bbb R}$ be a symmetric function and $h: D\to {\Bbb R}$ be a convex increasing function. Then $M(S,\rho,m)\subseteq M(S,h(\rho),m).$
\end{thm}
\begin{proof} Assume the contrary. Then there exist $X\in M(S,\rho,m)$ and $Y\subset S, |Y|=m,$ such that $$R_{h(\rho)}(Y)\lw R_{h(\rho)}(X).$$
Note that $f=-h^{-1}$ is a convex decreasing function, and the majorization theorem yields $$R_{\rho}(Y)\lw R_{\rho}(X).$$
The latter contradicts our assumption that $X\in M(S,\rho,m)$.   
\end{proof}

\medskip  

In this paper we consider the case when $S$ is the standard unit sphere ${\mathbb S}^{n-1}$ in ${\Bbb R}^{n}$. There are two natural distances on $\mathbb{S}^{n-1}$: the Euclidean distance $r$ and the angular distance $\varphi$. Here $r(x,y)$ denotes the Euclidian distance $||x-y||$ between two points $x, y\in {\Bbb S}^{n-1}$, while 
 $\varphi(x,y)$ denotes the angular distance in ${\Bbb S}^{n-1}$, i.e. $\varphi(x,y)=2\arcsin(||x-y||/2)$.
\begin{defn}  For $s\in{\mathbb R}$ define the function
$$ r_s(x,y):=\left\{
\begin{array}{l}
r^s(x,y), \; \text{ if } s>0;\\
\log{r(x,y)}, \; \text{ if } s=0;\\ 
-r^s(x,y), \; \text{ if } s<0.
\end{array} 
\right.
$$
\end{defn}	 

\begin{cor} \label{cor22} The following inclusions hold: 
\begin{itemize}
    \item[(i)] $M({\Bbb S}^{n-1},r_s,m)\subseteq M({\Bbb S}^{n-1},r_t,m)$, for all  $t\ge s$;
    \item[(ii)] $M({\Bbb S}^{n-1},r_s,m)\subseteq M({\Bbb S}^{n-1},\varphi,m)$, for all $s\in(-\infty, 1].$
\end{itemize}
\end{cor}

\begin{proof} Let $h(x)=x^{t/s},\;x>0$. Then $h(x)$ is a convex increasing function for all $t>s>0$. Since $r_t=h(r_s)$, Theorem \ref{T_3} implies (i) for $s>0$. The functions $h(x)=e^{sx}$, $s>0, x>0$, and $h(x)=\log(-x)/s$, $s<0, x<0$, prove the inclusions  
$$M({\Bbb S}^{n-1},r_0,m)\subseteq M({\Bbb S}^{n-1},r_s,m), \; s>0, \; \mbox{ and }  \; M({\Bbb S}^{n-1},r_s,m)\subseteq M({\Bbb S}^{n-1},r_0,m), \; s<0.$$
The remaining case $s<t<0$ in (i) follows from the fact that $h(x)=(-x)^{s/t},\; x<0,$ is a convex increasing function in $x$. 
It is clear that $h(x)=\arcsin(x/2), x\in [0,2],$ is a convex increasing function in $x$. Then Theorem \ref{T_3} also yields (ii).
\end{proof}

 Let $X=\{p_1,\ldots,p_m\}$ be an $m$--subset of ${\Bbb S}^{n-1}$ that consists of distinct points. Then the {\em Riesz $t$-energy} of $X$ is given by 
$$
E_t(X):= \sum\limits_{i< j}{\frac{1}{||p_i-p_j||^t}}, \text{ for } t>0, \text{ and } E_0(X):= \sum\limits_{i< j}{\log\left(\frac{1}{||p_i-p_j||}\right)}.  \eqno (2.1)
$$

Note that for $t=0$ minimizing $E_t$ is equivalent to maximizing $\prod\limits_{i\ne j}{||p_i-p_j||}$), which is Smale's  $7^{\rm th}$ problem \cite{Sm}. For $t=1$ we obtain the Thomson problem, and for $t\to\infty$ the minimum Riesz energy problem transforms into the Tammes problem.    

\medskip

Theorem \ref{T2} and  Corollary  \ref{cor22} yield:
\begin{cor} \label{cor3.1} Let $t\ge0$. If $X\subset{\Bbb S}^{n-1}$ gives the minimum of $E_t$ in the set of all $m$-subsets of\, ${\Bbb S}^{n-1}$, then $X\in M({\Bbb S}^{n-1},r_s,m)$ for all  $s> -t$.
\end{cor}

\section{Minima of majorizations}

Let $f$ be a convex function on ${\mathbb R}$.  Let $x_1,\ldots,x_n$ be a sequence of real numbers and $\bar x:=(x_1+\ldots+x_n)/n$. The {\em Jensen inequality} states that  
$$
f(\bar x)\le \frac{{f(x_1)+...+f(x_n)}}{n}.  
$$

If $ y\ge \bar x$, then it is easy to see that we have $$(y,\ldots,y)\lw (x_1,\ldots,x_n). \eqno (3.1) $$ 

Then the majorization theorem yields Jensen's inequality for a convex decreasing $f$: 
$$
f(y)\le \frac{{f(x_1)+...+f(x_n)}}{n}. 
$$

\medskip

In this section we extend the above inequality. First, we define a sequence $Y(T):=(y_1,\ldots,y_m)$ for any sequence of $m$ real numbers $T$, as follows. 

\begin{defn} Let $T=(T_1,\ldots,T_m)$, with  $T_1\le\ldots\le T_m$.   Let 
$$y_1(T):=\min\limits_{k=1,\ldots,m}{\frac{T_k}{k}},$$
$$y_2(T):=\min\limits_{k\ge2}{\frac{T_k-y_1(T)}{k-1}},$$
$$\ldots$$
$$y_m(T):=T_m-y_1(T)-\ldots-y_{m-1}(T),$$


 $$Y(T):=(y_1(T),\ldots,y_m(T)).$$ 
\end{defn}

Let a sequence $A=(a_1,\ldots,a_m)$ be such that 
$$
a_{(1)}+\ldots+a_{(i)}\le T_i, \; \mbox{ for all } i=1,\ldots,m,
$$
and $P(T_1,\ldots,T_m)$ denote the set of all such sequences.

\begin{lemma}\label{max_element}
If $T_1\le\ldots\le T_m$ and $A\in P(T_1,\ldots,T_m)$, then $Y(T_1,\ldots,T_m)\lw A$, i.e. $Y(T_1,\ldots,T_m)$ is the only maximum element in $P(T_1,\ldots,T_m)$.  
\end{lemma}
\begin{proof} Let $A\in P(T_1,\ldots,T_n)$. A proof immediately follows from the following inequalities:
$$a_{(1)}\le \frac{T_k}{k} \; \mbox{ for all } k=1,\ldots,m $$
and	 
$$a_{(i)}\le \frac{T_k-a_{(1)}-\ldots-a_{(i-1)}}{k-i+1} \; \mbox{ for all } k\ge i > 1. $$	
\end{proof}

\noindent{\bf Notation}. Given $S$, $\rho$, $n$, and $X\subset S$ with $|X|=n$, let $m:=n(n-1)/2$, and
$$Q_\rho(X):=(R_\rho(X))_\uparrow, \quad  (q_1,\ldots,q_m):=Q_\rho(X), \quad S_k^\rho(X):=q_1+\ldots+q_k, \quad k=1,\ldots,m.$$

\medskip

Note that Lemma \ref{max_element}, Theorem \ref{MajI} and (3.1) combined yield the following theorem. 

\begin{thm} \label{Tmin} Let $S$ be a set and $\rho:S\times S\to D\subset{\Bbb R}$ be a symmetric function. Let $T=(T_1,\ldots,T_m)$, where $m=n(n-1)/2$, be a sequence of real numbers with  $T_1\le\ldots\le T_m$ such that all $y_i(T)\in D$. Suppose that $X\subset S$ with $|X|=n$ satisfies
$$
S_k^\rho(X)\le T_k, \; k=1,\ldots,m.
$$  
Then  $Y(T)\lw R_\rho(X)$, and for every  convex decreasing function  $f:D\to{\Bbb R}$ we have $$E_f(X)\ge f(y_1(T))+\ldots+f(y_m(T)).$$

In particular, if there is $y\in D$ such that $T_k=ky$, for all $k=1,...,m$, then $y_k(T)=y$, for all $k=1,...,m$,  and we have 
$E_f(X)\ge mf(y)$. 
\end{thm}
\noindent{\bf Remark.} In the latter case, the inequality $E_f(X)\ge mf(y)$ is just Jensen's inequality.

\medskip

Now let us consider the case of the unit circle $S={\mathbb S}^1$ with angular distance $\rho=\varphi$. 

\begin{thm}\label{T32} Up to isometry, there exists a unique $M$-set of cardinality $n$ in the unit circle ${\Bbb S}^1$ with  $\rho=\varphi$: the vertices of a regular $n$-gon inscribed in ${\Bbb S}^1$. In other words, $M({\Bbb S}^1,\varphi,n)$ consists of the vertices of regular polygons.  
\end{thm}
\begin{proof} Let $X=\{p_1,\ldots,p_n\}\subset{\Bbb S}^1$ and $p_{n+i}:=p_i$ for all integer $i>0.$  We obviously have
$$ \sum\limits_{i=1}^n\varphi(p_i,p_{i+1})\le2\pi,$$
where the equality holds only if $|\angle{p_iOp_{i+1}}|=\varphi(p_i,p_{i+1})$ for all $i$. Moreover, we have 
$$ \sum\limits_{i=1}^n\varphi(p_i,p_{i+k})\le 2\pi k, \quad  k=1,2,..,\lfloor n/2\rfloor.$$
Then  $(3.1)$ yields 
$$
\pi_{n,k}:=(2\pi k/n,\ldots,2\pi k/n)\lw R_k:=(\varphi(p_1,p_{k+1}),\ldots,\varphi(p_n,p_{k})). 
$$

It is not hard to see that these inequalities  yield $$R_\varphi(\Pi_n)=\pi_{n,1}\cup \ldots\cup\pi_{n,\ell}\lw R_1\cup\ldots\cup R_\ell= R_\varphi(X),$$ where $\ell=\lfloor n/2\rfloor$ and $\Pi_n$ is the set of vertices of a regular $n$-gon  in ${\Bbb S}^1$. 
\end{proof}

 This theorem implies that $M({\Bbb S}^1,r_1,n)$ consists of the vertices of regular polygons. However, the set $M({\Bbb S}^1,r_2,n), n\ge 4,$ is much larger. In fact (see Section 5), $M({\Bbb S}^1,r_2,4)$ consists of the vertices of quadrilaterals with side lengths $2\pi-3\alpha$, $\alpha$, $\alpha$, $\alpha$ (in the angular measure), where $\pi/2\le \alpha \le 2\pi/3$. 

\medskip


\section{Optimal simplices and constrained $(n+k)$-sets}

First we show that Jensen's inequality for $(n+1)$-sets on $ {\Bbb S}^{n-1}$ yields optimality of regular simplices.

\begin{thm}\label{Tsimplex} Let $s\le2$. Then $M({\Bbb S}^{n-1},r_s,n+1)$ consists of regular simplices.
\end{thm}
\begin{proof} Let $X=\{p_1,\ldots,p_{m}\}\subset {\Bbb S}^{n-1}$ and $t_{i,j}:= p_i\cdot p_j$. Then 
$$
\sum\limits_{i,j}{t_{i,j}}=  \displaystyle{\left\| \sum_{i=1}^m p_i \right\|^2 } \ge 0.  \eqno (4.1)
$$	
Since $t_{i,i}=p_i\cdot p_i=1$, we have 
$$
\sum\limits_{i\ne j}{t_{i,j}}\ge -m.
$$

It is easy to see that $r_2(x,y)=||x-y||^2=2-2x\cdot y, \; x,y\in {\Bbb S}^{n-1}$. Then 
$$
\sum\limits_{i<j}{r_2(p_i,p_j)}= \sum\limits_{i<j}{(2-2t_{i,j})}\le m^2.
$$

Therefore, by (3.1) we have
$$
(a_m,\ldots,a_m)  \lw R_{r_2}(X), \; \text{ for } a_m:=\frac{2m}{m-1}. $$ 

Note that for $m=n+1$, the side lengths of a regular $n$-simplex are equal to $\sqrt{a_m}$. This completes the proof.
\end{proof}	

\noindent{\bf An open problem.} The set $M_n^S:=M({\Bbb S}^{n-1},\varphi,n+1)$, $n\ge3$,  is not as simple to describe as in the case $\rho=r_2$. For example, consider the case $n=3$. Let us define a two-parametric family of tetrahedra $ABCD$ in ${\Bbb S}^2$. Let the opposite edges $AC$ and $BD$ of  $ABCD$ be of equal length and the angle between them be $\theta$. Let $X$ be the midpoint of $AC$ and $Y$ be the midpoint of $BD$. Suppose that $X$, $Y$  and $O$, which is the center of ${\Bbb S}^2$, are collinear. Then $ABCD$ is uniquely (up to isometry) defined by the parameters $a=|OX|=|OY|$  and $\theta$. Let $\Delta_{a,\theta}$ denote such a tetrahedron $ABCD$. Note that  $\Delta_{0,\pi/2}$ is a square inscribed into the unit circle, while $\Delta_{1/\sqrt{3},\pi/2}$  is a regular tetrahedron. 
 
\textit{We conjecture that $M_3^S$ consists of the vertices of all tetrahedra $\Delta_{a,\theta}$, for $a\in [0,1/\sqrt{3}]$ and $0<\theta\le\pi/2$.} 

\textit{More generally, it is an interesting problem to find $M_n^S$ for all $n$.} 

\medskip

Now let us apply Theorem \ref{Tsimplex} to $P\subset {\Bbb S}^{n-1}$, with $n+2\le |P| \le 2n$. Davenport and Haj\'os \cite{DH}, and, independently, Rankin \cite{Rankin} proved that if $P$ is a subset of ${\mathbb S}^{n-1}$ with $|P|\ge n+2$, then the minimum distance between the points in $P$ is at most $\sqrt{2}$. For the case $|P|=2n$, Rankin proved that $P$ is a regular cross-polytope. Later on, Wlodzimierz Kuperberg \cite{kuperberg2007} extended this theorem. 
 
\medskip
 
\noindent {\bf Kuperberg's theorem.} {\em Let  $P$ be a
 ($n + k)$-point subset of the unit $n$-ball\, ${\mathbb B}^n$ with  $2\le k \le n$ such that the minimum distance between points in $P$ is at least $\sqrt{2}$. Then:
\begin{itemize}
     \item[(1)] every point of $P$ lies on the boundary of\, ${\mathbb B}^n$; 
     \item[~] and
     \item[(2)] ${\Bbb R}^n$ splits into the orthogonal product $\prod_{i=1}^k{L_i}$ of nondegenerate linear subspaces $L_i$ such that for  $S_i:=P\cap L_i$ we have $|S_i|=d_i+1$ and $\rk(S_i)=d_i$  $(i = 1, 2, . . . , k)$, where  $d_i:= \dim{L_i}$.
\end{itemize} }
 
\medskip 

With this above fact in mind, let us extend Definition \ref{def21}.   Let $S\subset{\mathbb R}^n$ and  $\rho:S\times S\to {\Bbb R}$ be a symmetric function. Then, let $\Omega=\Omega(S,\rho,q_0,m)$ denote the set of all  $X\subset S$ of cardinality $m$, such that for all distinct points $x,y\in X$ we have $||x-y||\ge q_0$. Finally, let $M(S,\rho,q_0,m)$ denote the set of all $X$ in $\Omega$ such that 
for any $Y\in\Omega$ either $R_\rho(X)\lw R_\rho(Y)$, or  $X$ and $Y$ are incomparable.      

\begin{thm}\label{th42} Let $2\le k\le n$ and $s\le2$. Then $M({\Bbb B}^{n},r_s,\sqrt{2},n+k) =M({\Bbb S}^{n-1},r_s,\sqrt{2},n+k)$ and this set consists of $k$ mutuallu orthogonal regular $d_i$-simplices $S_i$, such that all $d_i\ge 1$ and $d_1+...+d_k=n$.
\end{thm}
\begin{proof} By Kuperberg's theorem, we obtain that if $P\in M({\Bbb B}^{n},r_s,\sqrt{2},n+k)$, then (1) $P\subset{\Bbb S}^{n-1}$ and 
(2) $P$ consists of mutually orthogonal $d_i$-simplices $S_i$. By  Theorem \ref{Tsimplex}, all $S_i$ have to be regular. 
\end{proof}

\noindent{\bf Remarks.} \begin{enumerate} 
\item  From Rankin's theorem \cite{Rankin} it follows that  $\Omega({\Bbb S}^{n-1},r_s,\sqrt{2},2n)$ contains only regular cross-polytopes. However, if $2\le k<n-1$ then $\Omega({\Bbb S}^{n-1},r_s,\sqrt{2},n+k)$ contains infinitely many non-isometric point sets $P$ of several combinatorial types. For instance, if $k=2$ and $n=4$ then the respective dimensions $(d_1,d_2)$, as defined in the statement of Kuperberg's theorem, can be $(1,3)$ or $(2,2)$. 

\item An interesting open problem is to find $M({\Bbb S}^{n-1},r_s,n+k)$. Even for the case $k=2$, $n=3$ this seems a rather complicated task, see the discussion in Section 5.3. 

\item Recently, in our joint paper with Peter Dragnev \cite{DM}, we enumerated and classified all stationary logarithmic configurations of $n+2$ points in ${\Bbb S}^{n-1}$.  In particular, we showed that the logarithmic energy attains its relative minima at configurations that consist of two mutually orthogonal regular simplices. Actually, these configurations are the same as in Theorem \ref{th42} for $k=2$. 

 Now, let  $k\in[2,n]$.  Then, our conjecture is that the logarithmic energy of $n+k$ points in ${\Bbb S}^{n-1}$ attains its relative minima at configurations that consist of $k$ mutually orthogonal regular simplices. So far this conjecture remains open for $k = 3,...,n-1$.

\end{enumerate}


\section{Spherical $M$-sets of small cardinality}

In this section we consider spherical $M$-sets  of cardinality $m\le5$.  Clearly, for any $S$ and $\rho$ the case $m=2$ is trivial: $M(S,\rho,2)$ consists of pairs $(x,y)$ such that $\rho(x,y)$ attains its maximum on $S\times S$. However, the structure of $M$-sets for  $m>2$ is not so simple. 

\subsection{Spherical three-point $M$-sets} Theorems \ref{T32} and \ref{Tsimplex} yield that $M({\Bbb S}^1,\varphi,3)$ and $M({\Bbb S}^1,r_2,3)$ contain only the vertices of regular triangles. Let us now investigate $M({\Bbb S}^1,r_s,3)$ for all $s$. 

Consider the equation 
$$
(1-t)^z+2^{z-1}(1-t^2)^z=\left(\frac{3}{2}\right)^{z+1}, \; z=\frac{s}{2}. \eqno (5.1)
$$

For all $s$, this equation has a solution $t=-1/2$. It can be shown that  if   
$$4> s\ge s_0:=\log_{4/3}{(9/4)}\approx 2.8188,$$
 then (5.1) has one more solution $t_s\in(-1,-1/2)$.  Note that 
$$
t_{s_0}=-1, \quad t_4=-1/2, 
$$
and $t_s$ is increasing on the interval $[s_0,4]$ as a function of $s$. 

\begin{thm} \label{T51}The following cases hold for $M:=M({\Bbb S}^1,r_s,3)$:

\begin{enumerate}

\item if  $s\le\log_{4/3}{(9/4)}$, then $M$ contains only vertices of regular triangles;

\item if \,  $\log_{4/3}{(9/4)}<s<4$, then $M$ consists of the vertices of regular triangles and triangles with angular side lengths $\alpha$, $\alpha$, $2\pi-2\alpha$, where  $\alpha\in(\arccos(t_s),\pi]$;

\item if $s\ge 4$, then $M$ consists of the vertices of regular triangles and triangles with angular side lengths $\alpha$, $\alpha$, $2\pi-2\alpha$, $\alpha\in[2\pi/3,\pi]$.

\end{enumerate}
\end{thm} 

\begin{proof}
Suppose that we have a triangle $T$ inscribed in the unit circle ${\mathbb S}^1$ with angles $u_1, u_2, u_3$ and (Euclidean) side lengths $x_1,x_2,x_3$, where $u_1+u_2+u_3=\pi$.  Moreover, we assume that $u_1\ge u_2\ge u_3$. 

First, let us show that if $T$ is an $M$-set with $\rho=r_s$, then $u_1=u_2$. Indeed, fix $u_3$ so that $x_3=\sqrt{2-2\cos{2u_3}}$ is also fixed.  Then we have to maximize the function
$$F(u_1,u_2):=x_1^s+x_2^s$$ 
subject to $u_1+u_2=\pi-u_3$. 

If $u_3=0$, then we obviously have $u_1=u_2=\pi/2$. Assume that $u_3>0$. By the law of sines we get
$$x_1=c\sin{u_1}, \quad x_2=c\sin{u_2}, \quad c:=\frac{x_3}{\sin{u_3}}.$$
Then 
$$
F(u_1,u_2)=c^s(\sin^s{u_1}+\sin^s{u_2}). 
$$

The method of Lagrange multipliers gives the equality $\sin{u_1}=\sin{u_2}$ that under our constraints yields $u_1=u_2$. 

Now, for $T$, we have that $u_1=u_2=u$ and $u_3=\pi-2u$.  Therefore, 
  $$f_s(t):=(x_1^s+x_2^s+x_3^s)/2^{z+1}=(1-t)^z+2^{z-1}(1-t^2)^z , \quad t:=\cos{2u}. 
$$

Note that (5.1) is the equation $f_s(t)=f_s(-1/2)$. Since  $u\in[\pi/3,\pi/2]$, we have $t\in[-1,-1/2]$. It not hard to see that $f_s(t)\le f_s(-1/2)$ for  $0<s\le\log_{4/3}{(9/4)}$ and all $t$;  if \,  $\log_{4/3}{(9/4)}<s<4$, then $f_s(t)\le f_s(-1/2)$ for $t\in[t_s,-1/2]$; and if $s\ge4$ and $t\in[-1,-1/2)$, then $f_s(t)> f_s(-1/2)$. These observations complete the proof.  
\end{proof}

\subsection{Spherical four-point $M$-sets}
  
Theorem \ref{T32} yields that $M({\Bbb S}^1,\varphi,4)$  contains only vertices of squares.  This fact together with Corollary \ref{cor22}(ii) imply that $M({\Bbb S}^1,r_s,4)$ for $s\le 1$ also contains only vertices of squares.  

An interesting open problem is to describe  $M({\Bbb S}^1,r_s,4)$ for all $s$.  Let us mention that it can be proven that  $M({\Bbb S}^1,r_2,4)$ consists of the vertices of quadrilaterals inscribed into the unit circle with angular side lengths $\alpha$, $\alpha$, $\alpha$, $2\pi-3\alpha$, where $\pi/2\le \alpha \le 2\pi/3$.  
 
\medskip
 
 Theorem \ref{Tsimplex} yields that $M({\Bbb S}^2,r_s,4)$, with $s\le 2$, contains only vertices of regular tetrahedra. Another interesting problem is to describe what we have for the case $s>2$. 
 
 
 \subsection{Spherical  five-point $M$-sets}  From Theorem \ref{T32} we obtain that the sets $M({\Bbb S}^1,\varphi,5)$ and $M({\Bbb S}^1,r_s,5)$ for $s\le 1$  contain only vertices of regular pentagons.  
 
 The {\em triangular bi-pyramid} (or TBP, for short) is the configuration of $5$ points in ${\Bbb S}^2$ placed as follows: one point at the North pole, another one at the South pole, while the remaining three are arranged in an equilateral triangle on the equator.
 Theorem \ref{th42} yields that $M({\Bbb S}^2,r_s,\sqrt{2},5)$, $s\le2$, contains only the TBP. Moreover, the same result holds for $M({\Bbb S}^2,\varphi,\sqrt{2},5)$. Indeed, from Kuperberg's theorem it follows that $P$ consists of a 1-dimensional simplex $S_1$ that is a pair of antipodal points in ${\Bbb S}^2$, say the North and South poles, and a triangle $S_2$ on the equator. By  Theorem \ref{T32} this triangle has to be regular, i.e. $P$ is the TBP.
 
 The last known case is  $M({\Bbb S}^3,r_s,5)$ with $s\le 2$ that contains only vertices of regular 4-simplices. This follows from Theorem \ref{Tsimplex}. 
 
It is an interesting open problem to find $M({\Bbb S}^2,r_s,5)$.  By Corollary \ref{cor3.1} the global minimizer of the Riesz potential $E_t$ of 5 points lies in $M({\Bbb S}^2,r_s,5)$ for all $s>-t$. Then a solution to this  problem for some $s$ can help to find minimizers of $E_t$  for all $t>-s$. It is proved that the TBP is the minimizer of $E_t$ for $t=0$ \cite{DLT}, for $t=1,2$ \cite{Sch}, and for $t<15.048$ \cite{Sch2}. Note that the TBP is not the global minimizer for $E_t$ when  $t>15.04081$ \cite{MKS}. 
 

\section{Spherical $f$-designs}

In this section we define and study spherical $f$-designs.  In particular, we discuss possible relations between $f$-designs and $M$-sets, $\tau$-designs, and two-distance sets. Moreover, we extend Theorem \ref{Tsimplex} about the optimality of simplices, proving Theorem \ref{T42} below. Since $f$-designs are extremal spherical configurations, we believe that there are more connections between them and $M$-sets.  

\subsection{Definition of $f$-design}
Since a long time {\em Delsarte's method} (also known in coding theory as the {\em Linear Programming Bound}) has been widely used for finding cardinality bounds for codes (see \cite[Chap. 9,13]{CS} and \cite{DGS,Kab, lev98}). This approach for energy bounds was first applied by Yudin \cite{Y92}, then by Cohn and Kumar \cite{CK07}, and recently in \cite{BDHSS}.

In our case, this method relies on the positive semidefinite property of Gegenabauer polynomials $G_k^{(n)}(t)$  that can be defined via the following recurrence formula:
$$G_0^{(n)}=1,\;\; G_1^{(n)}=t,\; \ldots,\; G_k^{(n)}=\frac {(2k+n-4)\,t\,G_{k-1}^{(n)}-(k-1)\,G_{k-2}^{(n)}} {k+n-3.}.$$
Alternatively, $\{G_k^{(n)}\}_k$ can be defined as a family of orthogonal polynomials on the interval $[-1,1]$, with respect to the weight function $\rho(t)=(1-t^2)^{(n-3)/2}$.

Let $P=\{p_1,\ldots,p_m\}$ be a finite subset of ${\mathbb S}^{n-1}$, in other words, $P$ is a finite set of unit vectors. We define the $k$-th moment of $P$ as
$$
M_k(P):=\sum\limits_{i=1}^m\sum\limits_{j=1}^m {{G_k^{(n)}(t_{i,j})}}, \text{ where } t_{i,j}:=p_i\cdot p_j=\cos(\varphi(p_i,p_j)).
$$

It is well--known that Gegenabauer polynomials are {\em positive definite}. A real function $f$ on $[-1,1]$  is called positive definite (p.d.) in ${\mathbb S}^{n-1}$ if for every finite subset $P=\{p_1,\ldots,p_m\}$ in ${\mathbb S}^{n-1}$ the matrix $\bigl(f(t_{i,j})\bigr)^m_{i,j=1}$ is positive semidefinite.

The p.d. property of  Gegenabauer polynomials  yields that
$$
M_k(P)\ge0 \; \mbox{ for all  }  \;  k=1,2,...  \eqno (6.1) 
$$

Since  $G_1^{(n)}(t)=t$, then the inequality (6.1) for $k=1$ gives (4.1). 

\medskip 

Let $f$ be a function on $[-1,1]$ such that its Gegenbauer series $\sum_{k=0}^\infty{f_kG_k^{(n)}}$ is well--defined. Throughout this section we assume that this series converges uniformly to $f$ on the whole interval $[-1,1]$. Then we can write 
$$
f(t)=\sum\limits_{k=0}^\infty {f_kG_k^{(n)}(t)} \; \mbox{ for all } \; t\in [-1,1]. 
$$
Note that $f$ is p.d. if and only if all its \textit{Gegenbauer coefficients} are non--negative: $f_k\ge0$. 

It is easy to see that for any $P=\{p_1,\ldots,p_m\}$ in ${\mathbb S}^{n-1}$ we have 
$$
S_f(P):=\sum\limits_{i=1}^m\sum\limits_{j=1}^m {{f(t_{i,j})}}=\sum\limits_{k=0}^\infty{f_kM_k(P)}.  \eqno (6.2)
$$

\begin{defn} \label{def61}
Let $P=\{p_1,\ldots,p_m\}$ be a finite subset of the unit sphere ${\mathbb S}^{n-1}$. Let $D(P)$ denote the set of all inner products that occur between distinct $p_i$'s in $P$.  

For a given function $f(t)=\sum_k {f_k\,G_k^{(n)}(t)}$, we say that $P$ is an {\em $f$-design} if it satisfies the following properties: 
 \begin{enumerate}
 \item for all $k>0$ with $f_k\ne0$, we have that $M_k(P) = 0$; 
 \item  $D(P)\subset Z_f$, where $Z_f$ denotes the set of all $t\in [-1,1)$ such that  $f(t)=0$.
\end{enumerate}

For a given $f$ we say that an $f$-design is {\em of degree} $d$ if $f$ is a polynomial of degree $d$.
\end{defn} 

\noindent{\bf Remark.}  Property (1) in the above definition is related to linear programming slackness conditions and the concept of harmonic indices \cite{AY,BOT,BDK, DGS, DS, Zhu17}. Let $K$ be a subset of ${\mathbb N}$. A subset $P\subset{\mathbb S}^{n-1}$ is called a {\em spherical design of harmonic index $K$} if for all $k\in K$ we have $M_k(P)=0$. 
 
Property (2) in the above definition is related to the concept of annihilating polynomial from \cite{lev98}. Below we show that (2) yields a tight property of harmonic indices. 

Note that for some $P$ several degrees are possible. For instance,  the cross--polytope is the second degree design with 
$f(t) = t(t + 1)$. However, it is an $f$--design of degree 3 with all $f(t) = (at + b)t(t + 1))$, where  $ab\ne0$,  see Proposition \ref{prop62}.


\subsection {Delsarte's bound and $f$-designs} 

 Let  $T$ be a subset of the interval  $[-1,1)$. A set of points $P$ in ${\mathbb S}^{n-1}$ of cardinality $m$ is called a  $T$-{\it spherical code} if for every pair  $(x,y)$ of distinct points in $P$ the inner product $x\cdot y\in T$. We wish to maximize the cardinality $m$ over all  $T$-spherical codes of fixed dimension $n$. The {\em Delsarte} (or {\em linear programing}) bound relates this maximization problem to a minimization problem for certain real function $f$ as follows (see \cite{DGS,Kab,{lev98}}): 

\medskip 

\noindent {\em  Let  $T\subset [-1,1)$. Let $f$ be a function on $[-1,1]$ with all $f_k\ge0$ such that $f(t)\le0$ for all $t\in T$. Then for every  $T$-spherical code of cardinality $m$ we have that 
$$
mf_0\le f(1) \eqno (6.3)
$$
} 

There are several known examples of $P$  and $T$ when the inequality (6.3) turns into equality, see \cite{CS,DGS,Kab,{lev98}}.  Now we consider $f$--designs that imply the equality $mf_0=f(1)$. 

\begin{lemma} \label{lem61} Let $f(t)=\sum_k {f_k\,G_k^{(n)}(t)}$ be a function on $[-1,1]$. 
\begin{enumerate}
 \item  If  $P\subset{\mathbb S}^{n-1}$ is such that $D(P)\subset Z_f$, then $S_f(P) = |P|\, f(1)$. 
\item If there is an $f$-design in ${\mathbb S}^{n-1}$ of cardinality $m$, then  $f(1) = m f_0.$ 
\end{enumerate}
\end{lemma}

\begin{proof} 
 1. Let $P=\{p_1,\ldots,p_m\}\subset{\mathbb S}^{n-1}$ with $D(P)\subset Z_f$.   Then $f(t_{i,j})=0$ for all $i\ne j$ and we have 
 $$S_f(P)=mf(1). \eqno (6.4)$$ 
 
2. Let $P$ be an $f$-design.  Since $f_kM_k(P)=0$ for all $k>0$ we have 
$$S_f(P)=
\sum_k{f_kM_k(P)} = f_0M_0(P) =  f_0m^2.$$ 
Thus, 
$$f(1)=mf_0. \eqno (6.5)$$ 
\end{proof}

Now we derive some conditions for $P$ to be an $f$-design.

 \begin{thm} \label{t62} Let $f(t)=\sum_k {f_k\,G_k^{(n)}(t)}$ be a function with all $f_k\ge0$.  Let $P\subset{\mathbb S}^{n-1}$ with $|P|=m$ be such that $D(P)\subset Z_f$. Then $P$ is an $f$-design if and only if\,  $f(1)=mf_0$.
\end{thm}

\begin{proof} If $P=\{p_1,\ldots,p_m\}$ is an $f$-design then by Lemma \ref{lem61} we have  $f(1)=mf_0$. 

Suppose  $f(1)=mf_0$.  Since $D(P)\subset Z_f$, by (6.4) we have $S_f(P)=mf(1)$. Moreover, by assumption,  $f_k\ge0$ for all $k$.  Then  Delsarte's bound (6.3) yields:  
$$mf(1)=S_f(P)=\sum_k{f_kM_k(P)}\ge m^2f_0.$$
From (6.1) it follows that $f_kM_k(P)\ge0$, for all $k$.  Then the equality  $f(1)=mf_0$ holds only if  $f_kM_k(P)=0$, for all $k>0$. This is exactly Property (1) in Definition 6.1.  \end{proof}

\subsection {Spherical $f$-designs and $M$-sets}
Now we show that there is a simple connection between $f$-designs and  $M$-sets. 

\begin{thm} \label{T42} Let $f(t)=\sum_k {f_k\,G_k^{(n)}(t)}$ be a function on $[-1,1]$ with all $f_k\ge0$.  Then any $f$-design in ${\mathbb S}^{n-1}$  is an $M$-set  with $\rho(x,y)=-f(x\cdot y)$. 
\end{thm}

\begin{proof} Let $\rho(x,y):=-f(x\cdot y)$, where  $x,y\in {\mathbb S}^{n-1}.$  For $Y=\{y_1,\ldots,y_m\}\subset{\mathbb S}^{n-1}$ define 
$$
G_f(Y):=\sum\limits_{i<j}{\rho(y_i,y_j)}. 
$$

If $P=\{p_1,\ldots,p_m\}$ is an $f$-design then the following equalities hold:
$$
f(1)=mf_0, \quad \rho(p_i,p_j)=0, \; \forall \; i\ne j, \quad G_f(P)=0. 
$$ 

It is easy to see that  for an arbitrary $Y\subset{\mathbb S}^{n-1}$, $|Y|=m$,  (6.1) implies
$$
S_f(Y)=\sum_k{f_kM_k(Y)}\ge f_0m^2.
$$
Thus
$$
G_f(Y)=\frac{mf(1)-S_f(Y)}{2}=\frac{f_0m^2-S_f(Y)}{2}\le0. 
$$ 
Finally, by (3.1) we have 
$$
R_\rho(P)=(0,\ldots.0)\lw R_\rho(Y). 
$$
This completes the proof. 
\end{proof}

\noindent{\bf Open problem.} Consider $f$ as in Theorem 6.2 with all $f_k\ge0$ and $f(1) = m f_0$. Then Theorems \ref{t62} and \ref{T42} yield that if $D(P)\subset Z_f$, then $P$ is an $f$-design and $P\in M({\Bbb S}^{n-1},-f,m)$. It is easy to prove that if $Y\in M({\Bbb S}^{n-1},-f,m)$, then $D(Y)\subset Z_f$. {\em Is it true that $Y$ is always isomorphic to $P$?} There are several cases when the answer is positive (see \cite{BS}). 


\subsection {Spherical $\tau$- and $f$-designs} 

A spherical $\tau$-design $P$  is a set of points in $\mathbb{S}^{n-1}$ such that
\begin{eqnarray*}
\frac{1}{\mu(\mathbb{S}^{n-1})} \int_{\mathbb{S}^{n-1}} F(x) d\mu(x)= \frac{1}{m} \sum_{x \in P} F(x), \text{ with } m=|P|,  
\end{eqnarray*}
($\mu(x)$ is the surface area measure) holds for all polynomials $F(x) $ of total degree at most $\tau$. Equivalently, {\em $P$ is a $\tau$-design if and only if  $M_k(P)=0$ for all $k=1,2,\ldots,\tau$} (see \cite{DGS,lev98}).

\medskip

The following proposition directly follows from the definition of $f$- and $\tau$-designs.
\begin{prop}\label{prop62} If $P\subset{\mathbb S}^{n-1}$ is a $\tau$-design and $|D(P)|\le \tau$, then $P$ is an $f$-design of degree $\tau$ with 
$$
f(t)=g(t)\prod\limits_{x\in D(P)}(t-x), \quad \deg{g}\le \tau-|D(P)|.  
$$ 
\end{prop} 

There are many examples of spherical $f$-designs.  Let $C$ be the set of vertices of a regular cross-polytope. Then $D(C)=\{0,-1\}$ and $C$ is a spherical 3-design.  If $f(t):=(at+b)\,t(t+1)$,   $a, b\in{\mathbb R}, $ 
 then Proposition \ref{prop62} yields that $C$ is an $f$-design of degree 3. 

This example can be extended for universally optimal configurations. Note that all known universally optimal spherical configurations $P$ are $\tau$-designs with $\tau>|D(P)|$ \cite{CK07}. Therefore, if $f$ is the same as in Proposition \ref{prop62}, then  $P$ is an  $f$-design.

\medskip

However, the set of $f$-designs is much larger than the set of universally optimal configurations.  Let $P$ be the set of vertices of a regular 24-cell  $P$ in ${\mathbb S}^3$. It is known that $P$ is not universally optimal \cite{cohn07a}. In this case  $P$ is a 5-design and $D(P)=\{\pm1/2, 0, -1\}$. Thus, if  $$f(t):=(at+b)\,(t^2-1/4)(t^2+t),$$ then  $P$ is an $f$-design for all real $a$ and $b$.  


\subsection {Spherical two-distance sets and  $f$-designs} 

A finite collection $P$ of unit vectors in ${\mathbb R}^{n}$  is called a {\em spherical two-distance set} if there are two real numbers $a$ and $b$ such that the inner product of each pair of distinct vectors from $P$ takes value either $a$ or $b$. In particular, if the inner products in $P$ satisfy the condition $a=-b,$ then $P$ is a {\em set of equiangular lines}. In this subsection we discuss $f$-designs that are two-distance sets. 

Let $P$ be an $f$-design of degree 2. Then $|D(P)|\le |Z_f|\le2$, i.e. $P$ is a two-distance set. 

\begin{prop} Let $f(t)=(t-a)(t-b)$, where $a,b\in [-1,1)$ and $a+b\ne 0$. Then $P$  in ${\mathbb S}^{n-1}$ is an $f$-design if and only if $P$ 
is a two-distance 2-design. 
\end{prop}
\begin{proof} We have 
$$
f(t)=t^2-(a+b)t+ab=\frac{n-1}{n}G_2^{(n)}-(a+b)\,G_1^{(n)}+ab+1/n=f_2G_2^{(n)}+f_1G_1^{(n)}+f_0. 
$$
Let $P$ be an $f$-design. Since $f_1\ne0$ and $f_2\ne0$, $P$ is a 2-design.  If $P$ is a two-distance 2-design with inner products $a$ and $b$ then, by Proposition \ref{prop62}, $P$ is an $f$-design.
\end{proof}

Actually, all two-distance 2-designs can be obtained from strongly regular graphs as shown in \cite[Theorem 1.2]{BOGY}. This gives a characterization of  $f$-designs of degree 2 with $a+b\ne0$. 

The case $a=-b$ when $f$--designs become sets of equiangular lines is also very interesting. Note that the connection between these sets and strongly regular graphs is well--known \cite{DGS}.

If $a=-b$, we get $f(t)=t^2-a^2$ and then $f_0=1/n-a^2$, $f_1=0$, $f_2=1-1/n$. In this case Delsarte's bound (6.3) becomes
$$
m\le\frac{f(1)}{f_0}=\frac{n(1-a^2)}{1-na^2}. 
$$
For sets of equiangular lines this inequality is known as the {\em relative bound} as opposed to the {\em absolute} (or {\em Gerzon}) bound (see \cite{LeS} and a recent improvement in \cite{glaz16}): 
$$
m\le\frac{n(n+1)}{2}. \eqno (6.5)
$$

We have that a set $P$ in  ${\mathbb S}^{n-1}$ with $|P|=m$ is an $f$-design, where $f(t)=t^2-a^2$, if and only if $D(P)=\{a,-a\}$ and 
$m(1-na^2)=n(1-a^2).$  There are several known particular cases. However, the problem of classification of these designs is yet unsolved. 

\medskip

Now let $f(t):=g(t)(t-a)(t-b)$.  We would like to find all $P$ in  ${\mathbb S}^{n-1}$ with $|P|=m$ and $D(P)=\{a,b\}$ that are  $f$-designs.

  Consider the case $a+b\ge0$. (For the case $a+b<0$, see \cite{mus09a,glaz16}.)   In \cite{mus09a} we proved that if $a+b\ge0$, then the absolute bound (6.5) holds.  Moreover, this bound is tight: for all $n\ge7$ there are {\em maximal}, i.e. with $m=n(n+1)/2$, two-distance sets. 

Let the unit vectors $e_1,\ldots,e_{n+1}$ form an orthogonal basis of ${\Bbb R}^{n+1}$.  Let $V_n$ be the set of points $e_i+e_j, \; 1\le i<j\le n+1.$ Since $V_n$ lies in the hyperplane $\sum^{n+1}_{k=1} {x_k}=2$, we see that it represents a spherical two-distance set in ${\Bbb R}^n$. The cardinality of $V_n$ is $n(n+1)/2$. 

Let us rescale $V_n$ such that its points lie on the unit sphere ${\mathbb S}^{n-1}$. Let $\Lambda_n$ denote the resulting set.  It is not hard to determine the respective distances $a$ and $b$: 
$$
a=\frac{n-3}{2(n-1)}\,, \quad b=\frac{-2}{n-1},\, \quad 
a+b=\frac{n-7}{2(n-1)}. 
$$ 
We see that for $n>7$, $|\Lambda_n|$  attains the upper bound for two-distance sets with $a+b>0$. 

In fact, $\Lambda_n$ is a maximal  $f$-design of degree 2. The following questions seems interesting: {\em Are there other maximal $f$-designs with $a+b>0$ of degree} $d\ge2$? 

As noted above, there is a correspondence between $f$-designs of degree $2$ and strongly regular graphs. Actually, every graph $G$ can be embedded as a spherical two-distance set (see \cite{mus2dist}). This raises the following question: {\em Which graphs are embeddable as $f$-designs?}

\medskip
\medskip

\noindent{\bf Acknowledgments.} I  wish to thank Eiichi Bannai, Alexander Barg, Peter Boyvalenkov and Alexander Kolpakov for helpful discussions and useful comments.

\medskip

\noindent  O. R. Musin \\ University of Texas Rio Grande Valley, School of Mathematical and
 Statistical Sciences
 \\
 Moscow Institute of Physics and Technology
 \\
 The Institute for Information Transmission Problems of RAS 
 
 \medskip
 
\noindent   {\it Mailing address:}  One West University Boulevard, Brownsville, TX, 78520, USA.
 
 \medskip

\noindent   {\it E-mail address:} oleg.musin@utrgv.edu

\end{document}